\documentclass[11pt]{article}

\usepackage[margin=1.05in]{geometry}
\usepackage{amsmath,amssymb,amsthm,mathtools}
\usepackage{booktabs}
\usepackage{microtype}
\usepackage[round,authoryear]{natbib}
\usepackage{xcolor}
\usepackage[colorlinks=true,linkcolor=blue!55!black,citecolor=blue!55!black]{hyperref}

\newtheorem{theorem}{Theorem}[section]
\newtheorem{lemma}[theorem]{Lemma}
\newtheorem{proposition}[theorem]{Proposition}

\newcommand{\lam}{\lambda}
\newcommand{\E}{\mathbb E}
\renewcommand{\P}{\mathbb P}
\newcommand{\1}{\mathbf 1}
\newcommand{\F}{\mathcal F}
\newcommand{\A}{\mathfrak A}
\newcommand{\cS}{\mathcal S}
\newcommand{\To}{T_{\mathrm o}}
\newcommand{\st}{\tau}
\newcommand{\norm}[1]{\lVert #1\rVert}
\newcommand{\dd}{\mathop{}\!\mathrm{d}}
\newcommand{\e}{e}
\newcommand{\R}{\mathbb R}

\title{\LARGE Weak-Type Bounds for Convolution on the Boolean Hypercube}
\author{
      Junwei Lu\thanks{Department of Biostatistics, Harvard T.H. Chan School of
	Public Health. Email: \texttt{junweilu@hsph.harvard.edu}.}
      \qquad
	Shengtao Guo
      \qquad
	Ethan X. Fang
}
\date{}

\begin{document}
\maketitle

\begin{abstract}
      Let $G$ be the Boolean hypercube which carries uniform measure $\lambda$, and let
      $T_\mu$ denote convolution by a finite positive measure $\mu$ on $G$.  For
      $
       \psi_\mu(u)=\sup\{u\lambda(\{T_\mu f\geq u\}):f\geq0,
       \norm f_1=1\},
      $
      we prove Talagrand's convolution conjecture \citep{Talagrand1989}: if
      $\mu_a=((1+a)\delta_1/2+(1-a)\delta_{-1}/2)^{\otimes n}$ and $0<a<1$,
      then $\psi_{\mu_a}(u)\leq C_a/\sqrt{\log u}$ for every $u > 1$ and $n\geq1$, where
      $C_a$ depends only on $a$.  The proof utilizes the reverse-heat and Boolean-bridge
framework of \citet{Chen2026} and the localized terminal-discrepancy method of
\citet{XiangZhang2026}. We introduce a new power coupling:  each reverse edge ratio is split into two geometric powers.  This
choice produces a switched exponential weight which restores the exact
reverse jump rate of the perturbed coordinate.  The resulting endpoint
comparison yields an anti-concentration profile estimate 
without the iterated-logarithmic factor. The proof was discovered by the Odin Automatic AI Research Agent. 
\end{abstract}

\section{Introduction}\label{sec:introduction}

Let $G=\{-1,1\}^{n}$, equipped with uniform measure $\lam$.  For a finite positive measure $\mu$ on $G$,
define
\[
 T_\mu f(x)=\int_G f(x\odot y)\,\dd\mu(y),
\]
where $x\odot y$ denotes coordinatewise multiplication.  In his study of
$L^1$ regularization by singular convolution, \citet{Talagrand1989} introduced
\[
 \psi_\mu(u)=\sup\left\{u\lam(\{T_\mu f\geq u\}):f\geq0,
                              \ \norm f_1=1\right\}
\]
and considered the biased-coin product measure
\[
 \mu_a=\left(\frac{1+a}{2}\delta_1+
                  \frac{1-a}{2}\delta_{-1}\right)^{\otimes n},
 \text{ for } 0<a<1.
\]
Talagrand conjectured that $\psi_{\mu_a}(u)\leq C_a/\sqrt{\log u}$ for every
$u\geq2$ (see \citet[Problem~2]{Talagrand1989} and \citet[Conjecture 6]{Talagrand2016}). The Gaussian analogue was first studied by \citet{ball2013l1}. \citet{EldanLee2018} established the bound up to
an iterated-logarithmic loss  and \citet{Lehec2016} improved the result in its optimal form.

The first dimension-free decay estimate for a fixed biased coin was proved by
\citet{Chen2026}.  Chen obtained a bound of order
$(1-a)^{-2}(\log\log u)^{3/2}/\sqrt{\log u}$.  His work
introduced the perturbed reverse-heat coupling, the stopped score-energy
estimate, the Boolean heat bridge, the compensated exponential processes, and
the time-smoothed anti-concentration profile.  The coupling of
\citet{Chen2026} uses a parameter
$\alpha\asymp\log\log u$.  It separates two stopping levels and also sets the
strength of the rate perturbation.  The separation must grow in order to make
exceptional paths rare.  The same growth then enters the discrepancy between
the two terminal laws.  By retaining the starting information layer throughout
the Duhamel and bridge-energy calculation, \citet{XiangZhang2026} reduced the
extra factor to $\log\log u$.  Their improvement isolates the terminal
comparison more efficiently, but it keeps the perturbation and the
approximate-monotonicity mechanism of \citet{Chen2026}; the factor
$\alpha$ therefore remains in each layerwise discrepancy.

In this paper, we remove the extra $\log\log u$ factor and prove the following weak-type bound.

\begin{theorem}\label{thm:main}
For every $0<a<1$, $u > 1$, and $n\geq1$, we have
\begin{equation}\label{eq:main-theorem}
 \psi_{\mu_a}(u)\lesssim
 \kappa_a^2\left(\frac{\log\kappa_a}{\kappa_a-1}\right)^{1/2}
  \frac1{\sqrt{\log u}}, \text{ where } \kappa_a=\frac{1+a}{1-a}.
\end{equation}
\end{theorem}

Our proof introduces a new power coupling.  At the coupling
start, we freeze an exponent inversely proportional to the remaining distance
from the target band and split each reverse edge ratio into two geometric
powers.  This splitting is chosen so that exponential reweighting restores
exactly the unperturbed reverse flip rate of the perturbed trajectory, while
the remaining drift is the negative of the weighted arithmetic--geometric
mean deficit.  The resulting switched power supermartingale therefore yields
a weighted terminal comparison for every nonnegative terminal test function.
This comparison replaces the approximate pathwise monotonicity at tail used by
\citet{Chen2026}.  Paths on which the perturbed trajectory ends several
logarithmic bands above the reverse trajectory are not excluded; instead, they
incur an exponentially increasing weight.  Testing the comparison on
successive unit bands gives a contracting future-band recurrence whose total
coefficient is strictly smaller than one.  The coupling error is controlled
by extending the localized Duhamel--bridge argument of
\citet{XiangZhang2026} to the predictable coefficients of the power coupling,
and the initial gap layers are averaged using the time-smoothed profile
estimate of \citet[Lemma~4]{Chen2026}.  Because the frozen exponent has a
fixed numerator, no growing stopping buffer is required, and the additional
$\log\log u$ factor disappears.

\vspace{1em}

\noindent\textbf{The role of AI in this proof.}
Odin Automatic AI Research Agent was used to discover the proof. The final proofs were reorganized by the authors.

\vspace{1em}

\noindent\textbf{Notation.}
 The
point obtained from $x$ by flipping coordinate $i$ is denoted by $\sigma_i x$.
The indicator of an event $A$ is $\1_A$.  For a c\`adl\`ag process $X$,
$X_{t-}$ denotes the left limit.  We write
$A\lesssim B$ when $A\leq CB$ for a universal constant $C$.

\section{Preliminaries}\label{sec:prelim}

\subsection{Heat flow and information bands}

Let $(P_t)_{t\geq0}$ be the Boolean heat semigroup,
$P_tf(x)=\E f(x\odot\xi_t)$, where $\xi_t\sim\mu_{\e^{-t}}$.
Equivalently, after identifying a function on $G$ with its multilinear
extension, $P_tf(x)=f(\e^{-t}x)$.  Put $t_a=-\log a$, so that
$T_{\mu_a}=P_{t_a}$.
We assume in the rest of the paper that $f:G\to(0,\infty)$ and $\norm f_1=1$. The general nonnegative case can be handled by the strictly positive approximation $(f+\varepsilon)/(1+\varepsilon)$ for $\varepsilon>0$. See the proof of Theorem~\ref{thm:main} for details.
Write
$
 f_t=P_tf, d\nu_t=f_t\,d\lambda,
$
and, for an interval $I\subset\R$, define the size-biased logarithmic
profile
$
 \A_t(I)=\nu_t (\log f_t\in I).
$

\subsection{Reverse heat process}

The reverse process is the one introduced in
\citet[Section~2.5]{Chen2026}.  Set
$
 \To=2, T=t_a+\To.
$
The time $\To$ is the observation time in the reverse clock: the identity
$T-\To=t_a$ ensures that the reverse process at time $\To$ represents the
target heat time $t_a$, and it can be set to be any value larger than $1$, so we choose $\To=2$ for convenience.  For $0\leq t\leq T$, write
\begin{equation}\label{eq:reverse-heat-process}
      F_t(x)=\log f_{T-t}(x),\qquad
      Y_i(t,x)=\frac{f_{T-t}(\sigma_i x)}{f_{T-t}(x)},\qquad
      S_i(t,x)=\frac{1-Y_i(t,x)}2.  
\end{equation}
Here $\sigma_i x$ is obtained from $x$ by flipping coordinate $i$.
Let $N^{[1]},\ldots,N^{[n]}$ be independent Poisson random measures on
$[0,T]\times(0,\infty)$ with intensity $\dd t\dd z$, independent of
$V_0\sim\nu_T$, and let $(\F_t)_{0\leq t\leq T}$ be the completed
filtration they generate.  Define $V$ by
\[
 \dd V_t=\sum_{i=1}^n(-2V_{t-}^{(i)}e_i)
 \int_0^\infty
 \1_{\{0<z\leq Y_i(t,V_{t-})/2\}}N^{[i]}(\dd t,\dd z).
\]
Here $V_{t-}$ is the left limit and $e_i$ is the $i$th coordinate vector.
This is the time reversal, over the horizon $T$, of the forward heat flow
started from the probability density $f$.  In particular,
$
 V_t\sim\nu_{T-t}$, for $t\in[0,T]$,
and hence $V_{\To}\sim\nu_{t_a}$, i.e., the reverse process at time $\To$ follows the law of the targeted density. 

\section{Main proof}\label{sec:main-proof}

Like Lemma~1 in \citet{Chen2026}, we establish the following anti-concentration profile estimate for the reverse heat process to prove Theorem~\ref{thm:main}. In the rest of the paper, we denote the logarithm of the tail level $u$ by $\ell=\log u>0$ for $u > 1$.

\begin{proposition}\label{prop:fixed-band}
For every $\ell>0$, we have
\begin{equation}\label{eq:fixed-band}
 \A_{t_a}((\ell,\ell+1])
 \lesssim
 \kappa_a^2\left(\frac{\log\kappa_a}{\kappa_a-1}\right)^{1/2}
 \frac1{\sqrt\ell}.
\end{equation}
\end{proposition}

The proof of Proposition~\ref{prop:fixed-band} is deferred to Section~\ref{sec:anti-concentration-profile-estimate}.
Following the standard argument in \citet{Chen2026}, we can directly prove Theorem~\ref{thm:main} from Proposition~\ref{prop:fixed-band}.

\begin{proof}[Proof of Theorem~\ref{thm:main}]
      First suppose that $f>0$ and $\norm f_1=1$.  Fix $1 < v < u$ and write
      $\ell_v=\log v$.  Since $d\nu_{t_a}=f_{t_a}\,d\lambda$,
      Proposition~\ref{prop:fixed-band} gives
      \begin{align*}
       \lambda(\{P_{t_a}f\geq u\})
       &\leq\lambda(\{P_{t_a}f>v\})
        =\int_{\{\log f_{t_a}>\ell_v\}}f_{t_a}^{-1}\,\dd\nu_{t_a}\\
       &\leq\frac1v\sum_{j\geq0}\e^{-j}
             \A_{t_a}((\ell_v+j,\ell_v+j+1])\lesssim
       \kappa_a^2\left(\frac{\log\kappa_a}{\kappa_a-1}\right)^{1/2}
       \frac1{v\sqrt{\log v}}.
      \end{align*}
      Here the last step uses
      $
       \sum_{j\geq0}\frac{\e^{-j}}{\sqrt{\ell_v+j}}
       \leq\frac1{\sqrt{\ell_v}}\sum_{j\geq0}\e^{-j}
       \lesssim\frac1{\sqrt{\ell_v}}.
      $
      Letting $v\uparrow u$ proves \eqref{eq:main-theorem} for strictly
      positive $f$ and every $u > 1$.  
      
      Now let $f\geq0$ with $\norm f_1=1$ and use the strictly positive
      approximation $f_\varepsilon = (f+\varepsilon)/(1+\varepsilon)$.  Put
      $u_\varepsilon=(u+\varepsilon)/(1+\varepsilon) > 1$.  Since
      $P_{t_a}f_\varepsilon=(P_{t_a}f+\varepsilon)/(1+\varepsilon)$,
      $
       \{P_{t_a}f\geq u\}
       =\{P_{t_a}f_\varepsilon\geq u_\varepsilon\}.
      $
      Applying the strictly positive estimate to $f_\varepsilon$ at level
      $u_\varepsilon$, multiplying by $u$, and letting
      $\varepsilon\downarrow0$ proves \eqref{eq:main-theorem} for every
      nonnegative $f$ in the definition of $\psi_{\mu_a}$.
\end{proof}

To prove Proposition~\ref{prop:fixed-band}, we construct a process $W$ coupled to $V$. \citet{Chen2026} targeted to construct the coupling $(V_t,W_t)$ such that (1) the total variation distance between $\mathcal{L}(V_{\To})$ and $\mathcal{L}(W_{\To})$ is small, and (2) the approximate monotonicity at tail that $F_{\To}(V_{\To})$ is greater than $F_{\To}(W_{\To})$ at the tail level $\log u$. However, to achieve the bound in \eqref{eq:fixed-band}, the coupling needs to satisfy stronger conditions (see Lemmas~\ref{lem:discrepancy-summary} and \ref{lem:band-contraction-summary} below). We will introduce a power-coupling construction in Section~\ref{sec:power-coupling} to achieve these stronger conditions.

Fix a target level $u > 1$ with $\ell=\log u>0$, and a starting time
$\theta\in[\To-1,\To)$, which is the time $W_t$ starts to be perturbed from $V_t$.  Define the stopping time
\[
 \st=\inf\{t\in[\theta,\To]:F_t(V_t)\geq\ell+1\}\wedge\To
\]
and fix the universal constant $\alpha=5$.  Define
\begin{equation}\label{eq:R-theta-definition}
      R_\theta=[\ell-F_\theta(V_\theta)]_+,
      \qquad
      \mathcal E_\theta=\{R_\theta\geq2\alpha\},
      \qquad
      \bar\delta=\frac{\alpha\1_{\mathcal E_\theta}}{R_\theta+1}.  
\end{equation}
Compared to the construction in \citet{Chen2026}, our stopping time $\st$ is defined only based on  $V_t$ and does not depend on $\alpha$. Besides, the level $\alpha$ is set to be the constant in contrast to the choice of $\alpha \asymp \log \log u$ in \citet{Chen2026} which yields the additional $\log \log$ factor.

For $A\in\F_\theta$ with $A\subseteq\mathcal E_\theta$, define
\[
 \cS_A=\E\left[\1_A\bar\delta^2\int_\theta^\st
               \sum_{i=1}^n S_i(t,V_{t-})^2\,\dd t\right]
\text{ and }
 D_A=\sup_{B\subseteq G}
 \left|\P(A,W_{\To}\in B)-\P(A,V_{\To}\in B)\right|.
\]
Here $\cS_A$ is the stopped score energy generalizing Eq.~(21) in \citet{Chen2026} by constraining to an event $A$ and the localized total variation $D_A$ generalizes the total variation distance also by constraining to $A$. 
Recalling that $\kappa_a=(1+a)/(1-a)$, we denote two constants related to $a$ for the later use:
\[
 \Lambda_a=\frac{\kappa_a\log\kappa_a}{\kappa_a-1},
 \qquad
 K_a=\kappa_a^{3/2}\sqrt{\Lambda_a}
     =\kappa_a^2
       \left(\frac{\log\kappa_a}{\kappa_a-1}\right)^{1/2}.
\]

The first key result is to strengthen the total variation distance estimate for the coupling $(V_t,W_t)$ to the following lemma. Its proof is in Section~\ref{sec:discrepancy} which follows the
Duhamel--bridge--energy strategy of \citet[Section~3]{XiangZhang2026}.

\begin{lemma}[Localized total variation distance]
\label{lem:discrepancy-summary}
For every $A\in\F_\theta$ contained in $\mathcal E_\theta$,
\begin{equation}\label{eq:discrepancy-summary}
 D_A\lesssim
 \kappa_a\Lambda_a\sqrt{\cS_A\P(A)}
 +\kappa_a\Lambda_a^2\cS_A.
\end{equation}
\end{lemma}

Our second key condition is to replace the approximate monotonicity at tail as in \citet[Lemma 3]{Chen2026} for the coupling $(V_t,W_t)$ to the following lemma. Its proof is deferred to Section~\ref{sec:band-contraction} which is based on constructing a switched power supermartingale specified in Section~\ref{sec:power-supermartingale} such that we can have a stronger monotonicity property (see Proposition~\ref{prop:power-supermartingale}) for the coupling $(V_t,W_t)$ than \citet[Lemma 3]{Chen2026}.

\begin{lemma}
\label{lem:band-contraction-summary} For $r\in\R$, define
$
 A_\theta(r)=\P\left(\mathcal E_\theta,
                  F_{\To}(V_{\To})\in(r,r+1]\right).
$
With $c_0=\e^{1-\alpha}=\e^{-4}$, we have
\begin{equation}\label{eq:band-contraction-summary}
 (1-c_0)A_\theta(\ell)
 \leq c_0\sum_{j\geq1}\e^{-j}A_\theta(\ell+j)
      +2D_{\mathcal E_\theta}.
\end{equation}
\end{lemma}

The third key result is the stopped score estimate.  It is similar to \citet[Lemma~8]{Chen2026}, with the stopping time $\st$ changed. The proof of the following lemma is in Section~\ref{sec:stopped-score-energy}.

\begin{lemma}[Stopped score energy]\label{lem:score-summary}
We have
\[
 \E\left[\left.\int_\theta^\st\sum_{i=1}^n
 S_i(t,V_{t-})^2\,\dd t\right|\F_\theta\right]
 \leq\frac{\kappa_a-1}{\log\kappa_a}
       \bigl(R_\theta+1+\log\kappa_a\bigr).
\]
Moreover, for every $A\in\F_\theta$ contained in
$\mathcal E_\theta$,
\begin{equation}\label{eq:power-energy-summary}
 \cS_A\lesssim
 \frac{\kappa_a}{\Lambda_a}
 \E\left[\frac{\1_A}{R_\theta+1}\right]
 +(\kappa_a-1)
 \E\left[\frac{\1_A}{(R_\theta+1)^2}\right].
\end{equation}
\end{lemma}

The last key result is the time-smoothed anti-concentration profile estimate proved by \citet{Chen2026}.

\begin{lemma}[{\citet[Lemma 4]{Chen2026}}]
\label{lem:profile}
For every $\ell>2$, we have
\[
 \int_0^\infty\A_t((\ell,\ell+1])\,\dd t\lesssim\frac1\ell.
\]
\end{lemma}

\subsection{Anti-concentration profile estimate}\label{sec:anti-concentration-profile-estimate}

Combining Lemmas~\ref{lem:discrepancy-summary} -- \ref{lem:profile}, we can prove Proposition~\ref{prop:fixed-band} as follows.

\begin{proof}[Proof of Proposition~\ref{prop:fixed-band}]
Assume first that $\ell\geq C K_a^2$, where $C$ is a sufficiently large
universal constant. The elementary inequalities
$(\kappa_a-1)/\kappa_a\leq\log\kappa_a\leq\kappa_a-1$
give $1\leq\Lambda_a\leq\kappa_a$ and $K_a\geq1$. So we can choose sufficiently large $C$ such that $\ell\geq64$. We will discuss the small $\ell$ case in the end of the proof. The argument has three steps: (1) active discrepancy to control $D_{\mathcal E_\theta}$ mixing with the start time $\theta$, (2) inactive starting points to control $\overline{A}_\theta(\ell)$ mixing with the start time $\theta$, and (3) closing the recurrence using \eqref{eq:band-contraction-summary} to control our target on $\A_{t_a}((\ell,\ell+1])$.

\medskip
\noindent\emph{Step 1: active discrepancy.}
Let $m=\lfloor\ell/2\rfloor$ and decompose the active event into
\[
 E_r(\theta)=\{r\leq R_\theta<r+1\},
\text{ for }2\alpha\leq r<m,
 \text{ and }
 E_\infty(\theta)=\{R_\theta\geq m\}.
\]
Here and below $r$ ranges over integers.  From
\eqref{eq:power-energy-summary}, we have
\[
 \cS_{E_r(\theta)}\lesssim
 \left(\frac{\kappa_a}{\Lambda_a(r+1)}
       +\frac{\kappa_a-1}{(r+1)^2}\right)
 \P(E_r(\theta)).
\]
Substitution into \eqref{eq:discrepancy-summary}, together with
$\sqrt{x+y}\leq\sqrt x+\sqrt y$, and
$\sqrt{\kappa_a-1}\leq\kappa_a$, gives
\begin{equation}\label{eq:layer-discrepancy}
 D_{E_r(\theta)}\lesssim
 \left(\frac{K_a}{\sqrt{r+1}}
 +\frac{\kappa_a^2\Lambda_a}{r+1}
 +\frac{\kappa_a^2\Lambda_a^2}{(r+1)^2}\right)
 \P(E_r(\theta)).
\end{equation}
The same calculation, using $R_\theta\geq m\asymp\ell$ and
$\P(E_\infty(\theta))\leq1$, gives
\begin{equation}\label{eq:tail-layer-discrepancy}
 D_{E_\infty(\theta)}\lesssim
 \frac{K_a}{\sqrt\ell}
 +\frac{\kappa_a^2\Lambda_a}{\ell}
 +\frac{\kappa_a^2\Lambda_a^2}{\ell^2}.
\end{equation}

On $E_r(\theta)$ one has
$\ell-r-1<F_\theta(V_\theta)\leq\ell-r$.  Since
$F_\theta=\log f_{T-\theta}$ and $V_\theta\sim\nu_{T-\theta}$, the
change of variables $t=T-\theta$ gives
\begin{align*}
 \int_{\To-1}^{\To}\P(E_r(\theta))\,\dd\theta
 &\leq\int_{\To-1}^{\To}
 \A_{T-\theta}((\ell-r-1,\ell-r])\,\dd\theta =\int_{t_a}^{t_a+1}
 \A_t((\ell-r-1,\ell-r])\,\dd t.
\end{align*}
For $r<m$, the lower endpoint $\ell-r-1$ is comparable to $\ell$ and
exceeds $\To = 2$.  Lemma~\ref{lem:profile} therefore yields
\begin{equation}\label{eq:layer-mass-average}
 \int_{\To-1}^{\To}\P(E_r(\theta))\,\dd\theta\lesssim\frac1\ell.
\end{equation}

The events $E_r(\theta)$ and $E_\infty(\theta)$ form a disjoint partition
of $\mathcal E_\theta$.  Additivity of the corresponding signed terminal
measures, followed by the triangle inequality, gives
\[
 D_{\mathcal E_\theta}
 \leq\sum_{2\alpha\leq r<m}D_{E_r(\theta)}
      +D_{E_\infty(\theta)}.
\]
Combining this inequality with
\eqref{eq:layer-discrepancy}--\eqref{eq:layer-mass-average}, and using
\[
 \sum_{r\leq m}(r+1)^{-1/2}\lesssim\sqrt m,
 \quad
 \sum_{r\leq m}(r+1)^{-1}\lesssim\log(e m),
 \quad
 \sum_{r\geq0}(r+1)^{-2}\lesssim1,
\]
we obtain
\begin{equation}\label{eq:active-discrepancy-preabsorb}
 \int_{\To-1}^{\To}D_{\mathcal E_\theta}\,\dd\theta
 \lesssim
 \frac{K_a}{\sqrt\ell}
 +\frac{\kappa_a^2\Lambda_a\log(e\ell)}{\ell}
 +\frac{\kappa_a^2\Lambda_a^2}{\ell}.
\end{equation}
The last two terms are lower order when $\ell\geq C K_a^2$.
Indeed, $K_a^2=\kappa_a^3\Lambda_a$ and
$1\leq\Lambda_a\leq\kappa_a$.  Moreover,
$x\mapsto\log(e x)/\sqrt x$ decreases for $x\geq e$, so
\[
 \frac{\kappa_a^2\Lambda_a\log(e\ell)/\ell}
      {K_a/\sqrt\ell}
 \leq\frac{\log(eC K_a^2)}{\sqrt C\,\kappa_a}\lesssim1,
\text{ and }
 \frac{\kappa_a^2\Lambda_a^2/\ell}{K_a/\sqrt\ell}
 \leq\frac{\Lambda_a}{\sqrt C\,\kappa_a}\lesssim1.
\]
Thus, we can choose sufficiently large $C$ such that
\begin{equation}\label{eq:active-discrepancy-average}
 \int_{\To-1}^{\To}D_{\mathcal E_\theta}\,\dd\theta
 \lesssim\frac{K_a}{\sqrt\ell}.
\end{equation}

\medskip
\noindent\emph{Step 2: inactive starting points.}
Define
$
 \overline{A}_\theta(\ell)
 =\P\left(\mathcal E_\theta^c,
       F_{\To}(V_{\To})\in(\ell,\ell+1]\right).
$
Since $\mathcal E_\theta^c=\{F_\theta(V_\theta)>\ell-2\alpha\}$,
we first treat the near region
$\ell-2\alpha<F_\theta(V_\theta)\leq\ell+2$.  It is covered by the
finitely many unit intervals
$(\ell+q,\ell+q+1]$, for $q=-2\alpha,\ldots,1$.  For each such $q$, we have
\begin{align*}
 &\int_{\To-1}^{\To}\P\left(F_\theta(V_\theta)\in(\ell+q,\ell+q+1]\right)\,\dd\theta =\int_{t_a}^{t_a+1}\A_t((\ell+q,\ell+q+1])\,\dd t
 \lesssim\frac1\ell,
\end{align*}
where the last inequality follows from Lemma~\ref{lem:profile}; here
$\ell+q>2$ and $\ell+q\asymp\ell$.  Since the number of intervals depends
only on the fixed constant $\alpha$,
\begin{equation}\label{eq:inactive-near}
 \int_{\To-1}^{\To}\P\left(\ell-2\alpha<F_\theta(V_\theta)\leq\ell+2\right)\,\dd\theta
 \lesssim\frac1\ell.
\end{equation}

For the higher starting layers, we use the reciprocal reverse martingale
only at this point.
The time-dependent generator of $V$is
\begin{equation}\label{eq:reverse-generator}
 \widetilde{\mathcal L}_tg(x)
 =\frac12\sum_{i=1}^nY_i(t,x)\{g(\sigma_i x)-g(x)\}.
\end{equation}
The heat equation first gives
\begin{equation}\label{eq:F-time-derivative}
 \partial_tF_t(x)
 =-\frac{\partial_sf_s(x)}{f_s(x)}\bigg|_{s=T-t}
 =\frac12\sum_i(1-Y_i(t,x))=\sum_iS_i(t,x).
\end{equation}
Let $g_t(x)=\e^{-F_t(x)}=1/f_{T-t}(x)$.  Since
$g_t(\sigma_i x)=g_t(x)/Y_i(t,x)$, \eqref{eq:reverse-generator} and
\eqref{eq:F-time-derivative} give
\[
 (\partial_t+\widetilde{\mathcal L}_t)g_t(x)
 =-g_t(x)\sum_iS_i(t,x)
  +\frac{g_t(x)}2\sum_i(1-Y_i(t,x))=0.
\]
Thus the bounded process $\e^{-F_t(V_t)}$ is an $(\F_t)$-martingale.
Conditional Markov's inequality gives
\[
 \P\left(F_{\To}(V_{\To})\leq\ell+1\mid\F_\theta\right)
 \leq\e^{\ell+1-F_\theta(V_\theta)}.
\]
Consequently, on the event
$\{\ell+2+k<F_\theta(V_\theta)\leq\ell+3+k\}$, this conditional
probability is at most $\e^{-1-k}$.  Hence, we have
\begin{align*}
 &\int_{\To-1}^{\To}\P\left(\ell+2+k<F_\theta(V_\theta)\leq\ell+3+k,\,
                 F_{\To}(V_{\To})\in(\ell,\ell+1]\right)\,\dd\theta\\
 &\qquad\leq \e^{-1-k}
 \int_{\To-1}^{\To}\P\left(\ell+2+k<F_\theta(V_\theta)\leq\ell+3+k\right)\,\dd\theta\\
 &\qquad=\e^{-1-k}
 \int_{t_a}^{t_a+1}\A_t((\ell+2+k,\ell+3+k])\,\dd t \lesssim\frac{\e^{-1-k}}{\ell+2+k},
\end{align*}
where the last step is another application of Lemma~\ref{lem:profile}.
Summing over $k\geq0$ and using \eqref{eq:inactive-near}, we obtain
\begin{equation}\label{eq:inactive-average}
 \int_{\To-1}^{\To}\overline{A}_\theta(\ell)\,\dd\theta
 \lesssim\frac1\ell
 \lesssim\frac{K_a}{\sqrt\ell}.
\end{equation}

\medskip
\noindent\emph{Step 3: closing the recurrence.}
For each fixed $\theta$, the events $\mathcal E_\theta$ and
$\mathcal E_\theta^c$ partition the probability space.  Intersecting this
partition with the terminal band $(\ell,\ell+1]$ and using
$V_{\To}\sim\nu_{t_a}$ gives
\begin{equation}\label{eq:active-inactive-decomposition}
 \A_{t_a}((\ell,\ell+1])
 =A_\theta(\ell)+\overline{A}_\theta(\ell) \text{ and }
 A_\theta(\ell+j)
 \leq\A_{t_a}((\ell+j,\ell+j+1]),
 \text{ for all } j\geq0.
\end{equation}
Combining \eqref{eq:active-inactive-decomposition} with
Lemma~\ref{lem:band-contraction-summary}, we obtain, for every $\theta\in[\To-1,\To)$,
\[
 \A_{t_a}((\ell,\ell+1])
 \leq\overline{A}_\theta(\ell)
 +\frac{c_0}{1-c_0}\sum_{j\geq1}\e^{-j}
   \A_{t_a}((\ell+j,\ell+j+1])
 +\frac{2}{1-c_0}D_{\mathcal E_\theta}.
\]
Averaging over $\theta\in[\To-1,\To]$ and applying
\eqref{eq:active-discrepancy-average} and
\eqref{eq:inactive-average} yields
\begin{equation}\label{eq:band-recurrence}
 \A_{t_a}((\ell,\ell+1])
 \leq C\frac{K_a}{\sqrt\ell}
 +\frac{c_0}{1-c_0}\sum_{j\geq1}\e^{-j}
   \A_{t_a}((\ell+j,\ell+j+1]),
\end{equation}
where $C$ is a universal constant.  The total mass of the future-band operator is
\[
 \varrho_*=\frac{c_0}{1-c_0}\sum_{j\geq1}\e^{-j}
 =\frac{\e^{1-\alpha}}{(1-\e^{1-\alpha})(\e-1)}<1.
\]
Notice that $\varrho_* < 1$ and $\bar\delta < 1$ are the  constraints on the choice of $\alpha$ and $\alpha=5$ suffices.
For fixed positive $f$ on $G$, $\log f_{t_a}$ is bounded above, so
$\mathcal M_\ell:=\sup_{r\geq\ell}\sqrt r\,\A_{t_a}((r,r+1])<\infty$.
By \eqref{eq:band-recurrence}, we have
\begin{align*}
 \sqrt r\,\A_{t_a}((r,r+1])
 &\leq C K_a+\frac{c_0}{1-c_0}\sum_{j\geq1}\e^{-j}
   \sqrt r\,\A_{t_a}((r+j,r+j+1])\leq C K_a+\varrho_*\mathcal M_\ell.
\end{align*}
Taking the supremum over $r\geq\ell$ yields
$\mathcal M_\ell\leq C K_a+\varrho_*\mathcal M_\ell$.
Therefore $\mathcal M_\ell\lesssim K_a$.  This proves
\eqref{eq:fixed-band} for $\ell\geq C K_a^2$.

If $0<\ell<C K_a^2$, then
$\A_{t_a}((\ell,\ell+1])\leq1$ and
$1\lesssim K_a/\sqrt\ell$.  This completes the proof for every
$\ell>0$.
\end{proof}

\section{Power coupling and switched power supermartingale}
\label{sec:power-coupling}

In this section, we construct the power coupling $(V_t,W_t)$ and prove Lemmas~\ref{lem:discrepancy-summary} -- \ref{lem:score-summary}.

Condition on $\F_\theta$, so that $R_\theta$ and $\bar\delta$ are fixed.  For
$t\leq\To$, one has $T-t\geq t_a$. Recall that $Y_i(t,x)$ and $S_i(t,x)$ are defined in \eqref{eq:reverse-heat-process}. 

Starting from $W_\theta=V_\theta$, for $t\in[\theta,\To]$,
let
\begin{equation}\label{eq:W-SDE}
  \dd W_t=\sum_{i=1}^n(-2W_{t-}^{(i)}e_i)
 \int_0^\infty
 \1_{\{0<z\leq Y_i(t,V_{t-})/2+
       \delta_i(t,V_{t-})S_i(t,V_{t-})\1_{\{t\leq\st\}}\}}
 N^{[i]}(\dd t,\dd z).
\end{equation}
The process $\1_{\{t\leq\st\}}$ is always understood in its
left-continuous predictable version, so a jump crossing the barrier is
governed by the pre-stopping rates. Recall that $\bar\delta$ is defined in \eqref{eq:R-theta-definition} and $\bar\delta\in [0,1)$. We consider the ratio via power splitting
\[
 \delta_i(t,x)=
 \begin{cases}
 \displaystyle\frac{1-Y_i(t,x)^{\bar\delta}}{1-Y_i(t,x)},
     &Y_i(t,x)<1,\\[3mm]
 \bar\delta,&Y_i(t,x)=1,\\[2mm]
 \displaystyle\frac{Y_i(t,x)-Y_i(t,x)^{1-\bar\delta}}
 {Y_i(t,x)-1},&Y_i(t,x)>1.
 \end{cases}
\]
This is different from the ratio in \citet[Eq.~(18)]{Chen2026}. Our power coupling has the following interpretation.
 If $Y_i<1$, the $V$- and $W$-coordinates flip together at rate $Y_i/2$,
and $W$ has an additional flip at rate $(1-Y_i^{\bar\delta})/2$.  If
$Y_i\geq1$, their common rate is $Y_i^{1-\bar\delta}/2$, and $V$ has an
additional flip at rate $(Y_i-Y_i^{1-\bar\delta})/2$.  After $\st$, the two
coordinates flip synchronously at rate $Y_i/2$.  When
$\mathcal E_\theta$ fails, $\bar\delta=0$ and the construction reduces to
$W=V$. 
\citet[Lemma~5]{Chen2026} gives
\begin{equation}\label{eq:edge-bound}
 \kappa_a^{-1}\leq Y_i(t,x)\leq\kappa_a.
\end{equation}
Therefore, every coordinate
rate is bounded by a constant depending only on $a$.

For a function $h=h(x,y)$, write
$\Delta_i^yh(x,y)=h(x,\sigma_i y)-h(x,y)$ and
$\Delta_i^{xy}h(x,y)=h(\sigma_i x,\sigma_i y)-h(x,y)$.
The predictable generator is
\begin{align}
&\mathcal G_t=\overline{\mathcal L}_t^0+ \1_{\{t\leq\st\}}\mathcal B_t,
\label{eq:power-joint-generator}\\
&\overline{\mathcal L}_t^0h(x,y) =\frac12\sum_iY_i(t,x)\Delta_i^{xy}h(x,y)
\nonumber\\
&\mathcal B_th(x,y)
 =\sum_{S_i(t,x)>0}\delta_i(t,x)S_i(t,x)\Delta_i^yh(x,y)
 +\sum_{S_i(t,x)\leq0}\delta_i(t,x)S_i(t,x)
       \Delta_i^yh(\sigma_i x,y).
\label{eq:power-perturbation-generator}
\end{align}
The generator decomposition also implies that the $V$-coordinate retains its
unperturbed reverse transition law in the joint filtration.  This is stated
precisely in Lemma~\ref{lem:joint-filtration} and proved in
Section~\ref{subsec:joint-filtration-technical}.

\subsection{Localized total variation distance}\label{sec:discrepancy}

This section proves Lemma~\ref{lem:discrepancy-summary}.  The argument is the
localized Duhamel and Boolean-bridge method of
\citet[Section~3]{XiangZhang2026}, based on the bridge identities of
\citet[Lemmas~9--11]{Chen2026}.  The difference is that the perturbation
coefficients are generated by the power splitting rather than by the
one in \citet{Chen2026}.  

For $\zeta\in G$, define
\[
 H_t^\zeta(x)=\P(V_T=\zeta\mid V_t=x).
\]
By Lemma~\ref{lem:joint-filtration}, this terminal likelihood remains valid
in the enlarged filtration:
\begin{equation}\label{eq:joint-terminal-likelihood}
 H_t^\zeta(V_t)=\P(V_T=\zeta\mid\F_t),
 \qquad \theta\leq t\leq\To.
\end{equation}
Thus $H_t^\zeta(V_t)$ is the likelihood martingale used to condition the
joint process $(V,W)$ on $V_T=\zeta$.  Put
\[
 r_{t,i}^\zeta(x)=\frac{H_t^\zeta(\sigma_i x)}{H_t^\zeta(x)},
 \quad
 \lambda_{t,i}^\zeta(x)=
 \frac{1-\rho_tx_i\zeta_i}{1+\rho_tx_i\zeta_i},
 \text{ where } \rho_t=\e^{-(T-t)}.
\]
Then, we have the following relationships (see Section~\ref{sec:bridge-identities}):
\begin{equation}\label{eq:bridge-r-lambda}
 r_{t,i}^\zeta(x)Y_i(t,x)=\lambda_{t,i}^\zeta(x),
 \qquad
 \kappa_a^{-1}\leq\lambda_{t,i}^\zeta(x)\leq\kappa_a.
\end{equation}

Following \citet[Lemmas~9--11]{Chen2026}, let
\[
 \gamma_t=\e^{-(\To-t)},\qquad
 a_t=\frac{\gamma_t(1-a^2)}{1-a^2\gamma_t^2},
 \qquad
 b_t=\frac{a(1-\gamma_t^2)}{1-a^2\gamma_t^2},
\]
and, for $x,y,\zeta\in G$, set
$m_t^{[i]}(x,y,\zeta)=a_ty_i+b_tx_iy_i\zeta_i$.  If
$\phi:G\to\{0,1\}$ is identified with its multilinear extension, put
$q_t^\zeta(x,y)=\phi(m_t(x,y,\zeta))$.

We use two technical lemmas, proved in Sections~\ref{sec:bridge-identities} and~\ref{sec:conditioned-power} respectively.

\begin{lemma}[Boolean bridge identities]\label{lem:bridge-identities}
For $t<\To$, we have
\begin{align}
 \Delta_i^yq_t^\zeta(x,y)
 &=-2(a_ty_i+b_tx_iy_i\zeta_i)\partial_i\phi(m_t),
 \label{eq:bridge-diff-plus}\\
 \Delta_i^yq_t^\zeta(\sigma_i x,y)
 &=-2(a_ty_i-b_tx_iy_i\zeta_i)\partial_i\phi(m_t).
 \label{eq:bridge-diff-minus}
\end{align}
Moreover, we have the following control on the $b_t$:
\begin{equation}\label{eq:bridge-b-control}
 \lambda_{t,i}^\zeta(x)b_t^2
 \leq\frac{a^2}{1-a^2}\big(1-(m_t^{[i]}(x,y,\zeta))^2\big).
\end{equation}
Under the conditioned synchronized generator
$\mathcal L_t^{0,\zeta}h=\frac12\sum_i\lambda_{t,i}^\zeta\Delta_i^{xy}h$,
$q_t^\zeta$ is space-time harmonic and we have
\begin{equation}\label{eq:bridge-square-energy}
 (\partial_t+\mathcal L_t^{0,\zeta})(q_t^\zeta)^2
 =2a_t^2\sum_i\lambda_{t,i}^\zeta
          |\partial_i\phi(m_t)|^2.
\end{equation}
\end{lemma}

\begin{lemma}[Conditioned power coefficient]\label{lem:conditioned-power}
Under $\P^\zeta=\P(\,\cdot\mid V_T=\zeta)$, the predictable generator on
$[\theta,\To]$ is
$\mathcal L_t^{0,\zeta}+\1_{\{t\leq\st\}}\mathcal B_t^\zeta$, where
\begin{equation}\label{eq:conditioned-power-generator}
 \mathcal B_t^\zeta h(x,y)
 =\sum_{S_i(t,x)>0}\delta_i(t,x)S_i(t,x)\Delta_i^yh(x,y)
 +\sum_{S_i(t,x)\leq0}r_{t,i}^\zeta(x)\delta_i(t,x)S_i(t,x)
       \Delta_i^yh(\sigma_i x,y).
\end{equation}
The power coefficients satisfy
\begin{equation}\label{eq:conditioned-power-bound}
 \left|\delta_i(t,x)\left(
 \1_{\{S_i(t,x)>0\}}+r_{t,i}^\zeta(x)\1_{\{S_i(t,x)\leq0\}}
 \right)\right|^2
 \leq\kappa_a\Lambda_a^2\bar\delta^2\lambda_{t,i}^\zeta(x).
\end{equation}
\end{lemma}

\begin{proof}[Proof of Lemma~\ref{lem:discrepancy-summary}]
Fix an indicator $\phi:G\to\{0,1\}$ and define
\[
 U_t(x,y)=\sum_{\zeta\in G}H_t^\zeta(x)q_t^\zeta(x,y).
\]
For each $\zeta$, the likelihood $H_t^\zeta$ satisfies
$(\partial_t+\widetilde{\mathcal L}_t)H_t^\zeta=0$, and
$H_t^\zeta(\sigma_i x)=r_{t,i}^\zeta(x)H_t^\zeta(x)$.  Using
$r_{t,i}^\zeta Y_i=\lambda_{t,i}^\zeta$ and the harmonicity of
$q_t^\zeta$ in Lemma~\ref{lem:bridge-identities}, the product rule for
the synchronized jump generator gives
\begin{align*}
 &(\partial_t+\overline{\mathcal L}_t^0)
    \{H_t^\zeta(x)q_t^\zeta(x,y)\}=q_t^\zeta(x,y)
    (\partial_t+\widetilde{\mathcal L}_t)H_t^\zeta(x)
   +H_t^\zeta(x)
    (\partial_t+\mathcal L_t^{0,\zeta})q_t^\zeta(x,y)=0.
\end{align*}
Summing over $\zeta$, using 
$\sum_\zeta H_t^\zeta=1$, yields
$(\partial_t+\overline{\mathcal L}_t^0)U_t=0$ and
$U_{\To}(x,y)=\phi(y)$.

The synchronized bridge representation also gives, for every $x$ and
$\zeta$,
\begin{equation}\label{eq:bridge-starting-identity}
 q_\theta^\zeta(x,x)
 =\E[\phi(V_{\To})\mid V_\theta=x,V_T=\zeta].
\end{equation}
Since $W_\theta=V_\theta$, Bayes' formula and
Lemma~\ref{lem:joint-filtration} give, almost surely,
\[
 \E[\phi(V_{\To})\mid\F_\theta,V_T=\zeta]
 =\frac{P^V_{\theta,\To}(\phi H_{\To}^\zeta)(V_\theta)}
        {H_\theta^\zeta(V_\theta)}
 =q_\theta^\zeta(V_\theta,V_\theta),
\]
where the final equality is \eqref{eq:bridge-starting-identity}; the
denominator is positive under the standing assumption.  Together with
\eqref{eq:joint-terminal-likelihood} and the tower property, this gives
$
 U_\theta(V_\theta,W_\theta)
 =\E[\phi(V_{\To})\mid\F_\theta].
$
Dynkin's formula therefore yields the localized Duhamel identity
\begin{equation}\label{eq:localized-Duhamel-power}
 \E[\1_A\{\phi(W_{\To})-\phi(V_{\To})\}]
 =\E\left[\1_A\int_\theta^\st
       \mathcal B_tU_t(V_{t-},W_{t-})\,\dd t\right].
\end{equation}
This is the localization step of \citet{XiangZhang2026}; the multiplier
$\1_A$ may be inserted because $A\in\F_\theta$.

By Lemma~\ref{lem:conditioned-power},
$\mathcal B_tU_t=\sum_\zeta H_t^\zeta\mathcal B_t^\zeta q_t^\zeta$.
For fixed $\zeta$, let
\[
 c_{t,i}^\zeta=
 \begin{cases}
  a_ty_i+b_tx_iy_i\zeta_i,&S_i(t,x)>0,\\
  a_ty_i-b_tx_iy_i\zeta_i,&S_i(t,x)\leq0.
 \end{cases}
\]
Then
$|c_{t,i}^\zeta|^2\leq2(a_t^2+b_t^2)$.  Equations
\eqref{eq:bridge-diff-plus}--\eqref{eq:bridge-diff-minus} and
\eqref{eq:conditioned-power-bound} give, by Cauchy--Schwarz in $i$,
\begin{align*}
 |\mathcal B_t^\zeta q_t^\zeta(x,y)|
 &\leq2\sum_i|S_i(t,x)|
 \left|\delta_i(t,x)\left(
 \1_{\{S_i(t,x)>0\}}+r_{t,i}^\zeta(x)
 \1_{\{S_i(t,x)\leq0\}}\right)\right|
 |c_{t,i}^\zeta|\,|\partial_i\phi(m_t)|\\
 &\leq2\sqrt{2\kappa_a}\,\Lambda_a
 \left(\bar\delta^2\sum_iS_i(t,x)^2\right)^{1/2}
 \left(\sum_i\lambda_{t,i}^\zeta(x)(a_t^2+b_t^2)
 |\partial_i\phi(m_t)|^2\right)^{1/2}.
\end{align*}
A second Cauchy--Schwarz inequality in $\zeta$, with weights
$H_t^\zeta(x)$, yields
\begin{equation}\label{eq:pointwise-discrepancy-power}
 |\mathcal B_tU_t(x,y)|
 \lesssim\sqrt{\kappa_a}\,\Lambda_a
 \left(\bar\delta^2\sum_iS_i(t,x)^2\right)^{1/2}
 \Gamma_t(x,y)^{1/2}, \text{ where }
\end{equation}
\[
 \Gamma_t(x,y)=\sum_\zeta H_t^\zeta(x)
 \sum_i\lambda_{t,i}^\zeta(x)(a_t^2+b_t^2)
       |\partial_i\phi(m_t(x,y,\zeta))|^2.
\]
Thus \eqref{eq:localized-Duhamel-power} and Cauchy--Schwarz give
\begin{equation}\label{eq:D-via-power-energy}
 \left|\E[\1_A\{\phi(W_{\To})-\phi(V_{\To})\}]\right|
 \lesssim\sqrt{\kappa_a}\,\Lambda_a
 \cS_A^{1/2}Y_A^{1/2},
\end{equation}
where
$Y_A=\E[\1_A\int_\theta^{\To}\Gamma_t(V_t,W_t)\,\dd t]$.

Conditioning on $V_T$, split $Y_A=\Psi_a^A+\Psi_b^A$, where
\begin{align*}
 \Psi_a^A&=\E\left[\1_A\int_\theta^{\To}a_t^2
 \sum_i\lambda_{t,i}^{V_T}(V_t)
 |\partial_i\phi(m_t(V_t,W_t,V_T))|^2\,\dd t\right],\\
 \Psi_b^A&=\E\left[\1_A\int_\theta^{\To}b_t^2
 \sum_i\lambda_{t,i}^{V_T}(V_t)
 |\partial_i\phi(m_t(V_t,W_t,V_T))|^2\,\dd t\right].
\end{align*}
By Parseval's inequality for the first-level coefficients under the product measure with
coordinate means $z_i$ (see \citet[Chapter~8]{ODonnell2014} and also
\citet[Lemma~7]{Chen2026}), we have
\begin{equation}\label{eq:biased-level-one-used}
 \sum_i(1-z_i^2)|\partial_i\phi(z)|^2
 \leq\phi(z)-\phi(z)^2\leq\frac14.
\end{equation}
We apply \eqref{eq:biased-level-one-used} at
$z=m_t(V_t,W_t,V_T)\in[-1,1]^n$ for $t<\To$; the endpoint case follows by continuity.  Together with
\eqref{eq:bridge-b-control} and $\To-\theta\leq1$, this gives
\begin{equation}\label{eq:Psi-b-power}
 \Psi_b^A\lesssim\frac{a^2}{1-a^2}\P(A)
 \lesssim\kappa_a\P(A).
\end{equation}

For $0<\varepsilon<\To-\theta$, define
\begin{align*}
 \Psi_{a,\varepsilon}^A
 &=\E\left[\1_A\int_\theta^{\To-\varepsilon}a_t^2
 \sum_i\lambda_{t,i}^{V_T}(V_t)
 |\partial_i\phi(m_t(V_t,W_t,V_T))|^2\,\dd t\right],\\
 Y_{A,\varepsilon}
 &=\E\left[\1_A\int_\theta^{\To-\varepsilon}
 \Gamma_t(V_t,W_t)\,\dd t\right].
\end{align*}
Fix $\zeta\in G$ and write $\P^\zeta=\P(\,\cdot\mid V_T=\zeta)$.
The standing positivity assumption gives $\P(V_T=\zeta)>0$.  Under
$\P^\zeta$, apply Dynkin's formula on
$[\theta,\To-\varepsilon]$ to $(q_t^\zeta)^2$ with generator
$\mathcal L_t^{0,\zeta}+
\1_{\{t\leq\st\}}\mathcal B_t^\zeta$.  Since $A\in\F_\theta$, the
resulting martingale identity may be multiplied by $\1_A$.  Using
\eqref{eq:bridge-square-energy}, then multiplying by
$\P(V_T=\zeta)$ and summing over $\zeta$, gives
\begin{align}
 2\Psi_{a,\varepsilon}^A
 &=\E[\1_A(q_{\To-\varepsilon}^{V_T})^2]
   -\E[\1_A(q_\theta^{V_T})^2] -
 \E\left[\1_A\int_\theta^{\To-\varepsilon}
 \1_{\{t\leq\st\}}\mathcal B_t^{V_T}
       (q_t^{V_T})^2\,\dd t\right].
 \label{eq:Psi-a-Dynkin}
\end{align}
Here and below the state arguments are suppressed.  Since
$0\leq q_t^\zeta\leq1$, the first two terms on the right-hand side of
\eqref{eq:Psi-a-Dynkin} contribute at most $\P(A)$.  Moreover,
$|\Delta(q^2)|\leq2|\Delta q|$.  The coefficient allocation used in
\eqref{eq:pointwise-discrepancy-power}, followed by Cauchy--Schwarz,
therefore gives
\[
 \Psi_{a,\varepsilon}^A
 \lesssim\P(A)+
 \sqrt{\kappa_a}\,\Lambda_a
 \cS_A^{1/2}Y_{A,\varepsilon}^{1/2}.
\]
The integrands defining $\Psi_{a,\varepsilon}^A$ and
$Y_{A,\varepsilon}$ are nonnegative, so these quantities increase to
$\Psi_a^A$ and $Y_A$, respectively, as $\varepsilon\downarrow0$.
Furthermore, the Poisson random measures have no atom at the
deterministic time $\To$, so
$(V_{\To-\varepsilon},W_{\To-\varepsilon})\to(V_{\To},W_{\To})$
almost surely.  Thus dominated
convergence gives the convergence of the two boundary terms in
\eqref{eq:Psi-a-Dynkin}; hence no endpoint term is lost.  Letting
$\varepsilon\downarrow0$, combining with \eqref{eq:Psi-b-power}, and
applying Young's inequality yields
\begin{equation}\label{eq:power-energy-closure}
 Y_A\lesssim\kappa_a\P(A)+\kappa_a\Lambda_a^2\cS_A.
\end{equation}
Substituting \eqref{eq:power-energy-closure} into
\eqref{eq:D-via-power-energy} gives
\[
 \left|\E[\1_A\{\phi(W_{\To})-\phi(V_{\To})\}]\right|
 \lesssim
 \kappa_a\Lambda_a\sqrt{\cS_A\P(A)}
 +\kappa_a\Lambda_a^2\cS_A.
\]
Taking the supremum over indicator tests, with the total variation
convention stated after the definition of $D_A$, proves
\eqref{eq:discrepancy-summary}.
\end{proof}

\subsection{Switched power supermartingale}\label{sec:power-supermartingale}
We construct a switched power supermartingale in order to prove Lemma~\ref{lem:band-contraction-summary}  in Section~\ref{sec:band-contraction}.
For $t\in[\theta,\To]$, define the power gap
\begin{equation}\label{eq:power-gap}
 \Xi_t=
 \begin{cases}
 F_t(W_t)-(1-\bar\delta)F_t(V_t)-\bar\delta F_\theta(V_\theta),&t\leq\st,\\[1mm]
 F_t(W_t)-F_t(V_t)+\bar\delta(F_\st(V_\st)-F_\theta(V_\theta)),&t\geq\st.
 \end{cases}
\end{equation}
The two expressions agree at $t=\st$, including when the stopping barrier is
crossed by a jump.  Put $M_t=\e^{\Xi_t}$.  For a nonnegative terminal test $h:G\to[0,\infty)$, let $H$ be its
space-time harmonic extension for the unperturbed reverse process; equivalently,
$H$ is the unique bounded solution of the reverse backward equation with
terminal condition $H_{\To}=h$.

\begin{proposition}
\label{prop:power-supermartingale}
The process $M_tH_t(W_t)$ is a nonnegative supermartingale on
$[\theta,\To]$.  Consequently, for every nonnegative $h$ and every
$A\in\F_\theta$,
\begin{equation}\label{eq:power-weighted-comparison}
 \E[\1_AM_{\To}h(W_{\To})]
 \leq\E[\1_Ah(V_{\To})].
\end{equation}
\end{proposition}

\begin{proof}
Condition on $\F_\theta$.  By \eqref{eq:F-time-derivative},
$\partial_tF_t(x)=\sum_iS_i(t,x)$.  For later reference, put
\[
 X_i(t,y)=\frac{f_{T-t}(\sigma_i y)}{f_{T-t}(y)},\qquad
 \mathcal L_t^{W,\mathrm{rev}}g(y)
 =\frac12\sum_iX_i(t,y)\{g(\sigma_i y)-g(y)\}.
\]
Thus $(\partial_t+\mathcal L_t^{W,\mathrm{rev}})H_t=0$.
The continuous part of the drift can be computed coordinate by coordinate.
Write $Y=1-2S_i(t,V_{t-})$, $X=1-2S_i(t,W_{t-})$, and
$v=(1-Y)/2$, $w=(1-X)/2$.

Before $\st$, suppose first that $Y<1$.  The continuous relative coefficient
of $M$ is $w-(1-\bar\delta)v$.  A common flip has rate $Y/2$ and multiplier
$X/Y^{1-\bar\delta}$, while a $W$-only flip has rate $(1-Y^{\bar\delta})/2$ and multiplier $X$.
The coordinate contribution to the relative drift is therefore
\[
 w-(1-\bar\delta)v+\frac Y2\left(\frac X{Y^{1-\bar\delta}}-1\right)
 +\frac{1-Y^{\bar\delta}}{2}(X-1)
 =-\frac12\big((1-\bar\delta)+\bar\delta Y-Y^{\bar\delta}\big).
\]
If $Y\geq1$, a common flip has rate $Y^{1-\bar\delta}/2$ and multiplier
$X/Y^{1-\bar\delta}$, while a $V$-only flip has rate
$(Y-Y^{1-\bar\delta})/2$ and multiplier $Y^{-(1-\bar\delta)}$.  The same simplification gives
again
\begin{equation}\label{eq:power-drift-defect}
 -\frac12\big((1-\bar\delta)+\bar\delta Y-Y^{\bar\delta}\big)\leq0,
\end{equation}
where the sign follows from
$Y^{\bar\delta}\leq(1-\bar\delta)+\bar\delta Y$, valid by
$\bar\delta\in [0,1)$.

The exponential weighting also restores the exact reverse $W$-rate.  If
$Y<1$, the weighted coefficient of a $W$-flip is
\[
 \frac Y2\frac X{Y^{1-\bar\delta}}+\frac{1-Y^{\bar\delta}}{2}X=\frac X2.
\]
If $Y\geq1$, it is $Y^{1-\bar\delta}X/(2Y^{1-\bar\delta})=X/2$.
Thus, before $\st$, the full product-generator calculation is
\begin{align}
 \frac{(\partial_t+\mathcal G_t)
        \{M_tH_t(W_t)\}}{M_t}
 &= (\partial_t+\mathcal L_t^{W,\mathrm{rev}})H_t(W_t)\notag\\
 &\quad-\frac12\sum_i
 \bigl[(1-\bar\delta)+\bar\delta Y_i(t,V_t)
       -Y_i(t,V_t)^{\bar\delta}\bigr]H_t(W_t)\leq0.
 \label{eq:full-power-product-generator}
\end{align}
The first term vanishes by the backward equation.  Notice that the
marginal $W$-rate has not been replaced by the reverse rate; rather, the
coefficient of the $W$-difference in the weighted product generator is
exactly $X_i(t,W_t)/2$.

After $\st$, a common flip has rate $Y/2$ and multiplier $X/Y$; hence the
weighted $W$-rate is again $X/2$ and
\[
 \frac{(\partial_t+\mathcal G_t)
        \{M_tH_t(W_t)\}}{M_t}
 =(\partial_t+\mathcal L_t^{W,\mathrm{rev}})H_t(W_t)=0.
\]
The predictable stopping convention and the agreement of the two
expressions in \eqref{eq:power-gap} ensure that no additional jump or
finite-variation term occurs at $\st$.

Because $f>0$ on the finite state space, $F_t$, $M_t$, and $H_t$ are
bounded on the compact time interval, and all jump rates are bounded.
The compensated Dynkin formula therefore turns the nonpositive drift in
\eqref{eq:full-power-product-generator} directly into the asserted
supermartingale property, without localization.
At time $\theta$, $M_\theta=1$ and $W_\theta=V_\theta$.  Hence the
supermartingale inequality gives
\[
 \E[\1_AM_{\To}h(W_{\To})]
 \leq \E[\1_AH_\theta(V_\theta)].
\]
By Lemma~\ref{lem:joint-filtration},
$H_\theta(V_\theta)=\E[h(V_{\To})\mid\F_\theta]$.  Since
$A\in\F_\theta$, the right-hand side equals
$\E[\1_Ah(V_{\To})]$, proving
\eqref{eq:power-weighted-comparison}.
\end{proof}

\subsection{Band contraction}\label{sec:band-contraction}

\begin{proof}[Proof of Lemma~\ref{lem:band-contraction-summary}]
On the event
\[
 \mathcal E_\theta\cap
 \{F_{\To}(W_{\To})\in(\ell+j,\ell+j+1]\}
 \cap\{F_{\To}(V_{\To})\leq\ell+1\},
\]
one has $\Xi_{\To}>\alpha+j-1$.  Indeed, if $\st=\To$, then the first line of
\eqref{eq:power-gap} gives
\[
 \Xi_{\To}>
 \ell+j-(1-\bar\delta)(\ell+1)
 -\bar\delta(\ell-R_\theta)
 =j-1+\bar\delta(R_\theta+1).
\]
If $\st<\To$, then $F_\st(V_\st)\geq\ell+1$, and the second line gives the
same lower bound.  Since
$\bar\delta(R_\theta+1)=\alpha$ on $\mathcal E_\theta$,
Proposition~\ref{prop:power-supermartingale}, applied to the indicator of
$(\ell+j,\ell+j+1]$, yields
\begin{equation}\label{eq:band-weighted-crossing}
 \P\left(\mathcal E_\theta,
       F_{\To}(W_{\To})\in(\ell+j,\ell+j+1],
       F_{\To}(V_{\To})\leq\ell+1\right)
 \leq c_0\e^{-j}A_\theta(\ell+j).
\end{equation}

Introduce only the auxiliary zeroth-band mass
\[
 B_\theta(\ell)=\P\left(\mathcal E_\theta,
                     F_{\To}(W_{\To})\in(\ell,\ell+1]\right).
\]
The case $j=0$ of \eqref{eq:band-weighted-crossing} gives
\[
 B_\theta(\ell)
 \leq c_0A_\theta(\ell)
 +\P\left(\mathcal E_\theta,
       F_{\To}(V_{\To})>\ell+1\geq F_{\To}(W_{\To})\right).
\]
Testing the half-line $(\ell+1,\infty)$ in the definition of
$D_{\mathcal E_\theta}$ bounds the last probability by
\[
 \P\left(\mathcal E_\theta,
       F_{\To}(W_{\To})>\ell+1\geq F_{\To}(V_{\To})\right)
 +D_{\mathcal E_\theta}.
\]
Partitioning the first event into the disjoint bands
$(\ell+j,\ell+j+1]$, $j\geq1$, and applying
\eqref{eq:band-weighted-crossing}, we obtain
\[
 \P\left(\mathcal E_\theta,
       F_{\To}(W_{\To})>\ell+1\geq F_{\To}(V_{\To})\right)
 \leq c_0\sum_{j\geq1}\e^{-j}A_\theta(\ell+j).
\]
Finally, testing $(\ell,\ell+1]$ in the definition of
$D_{\mathcal E_\theta}$ gives
$A_\theta(\ell)\leq B_\theta(\ell)+D_{\mathcal E_\theta}$.
Combining these inequalities proves
\eqref{eq:band-contraction-summary}.
\end{proof}

\subsection{Stopped score energy}\label{sec:stopped-score-energy}

\begin{proof}[Proof of Lemma~\ref{lem:score-summary}]
The calculation is the one used in the proof of
\citet[Lemma~8]{Chen2026}; only the stopping barrier is different.  Along the
reverse process, the jump It\^o formula gives
\begin{equation}\label{eq:score-Ito-power}
 F_t(V_t)-F_\theta(V_\theta)
 -\int_\theta^t\frac12\sum_i
 \big(Y_i\log Y_i-Y_i+1\big)\,\dd r
\end{equation}
as a martingale.  Here $Y_i=Y_i(r,V_{r-})$.  For
$Y\in[\kappa_a^{-1},\kappa_a]$, the convexity estimate used in
\citet[Eq.~(40)]{Chen2026} yields
\[
 \frac12(Y\log Y-Y+1)
 \geq\frac{\log\kappa_a}{\kappa_a-1}S_i^2.
\]

We first verify the pathwise upper bound required for every starting
state.  If $F_\theta(V_\theta)\geq\ell+1$, then $\st=\theta$ and the
increment is zero.  If
$\ell<F_\theta(V_\theta)<\ell+1$, then $R_\theta=0$ and, before the
barrier is reached, the terminal increase is at most
$1+\log\kappa_a$.  Finally, if
$F_\theta(V_\theta)\leq\ell$, then
$F_\theta(V_\theta)=\ell-R_\theta$ and the increase is at most
$R_\theta+1+\log\kappa_a$.  Here a boundary-crossing jump has upward
size at most $\log\kappa_a$ by \eqref{eq:edge-bound}.  Thus, in all
cases,
\[
 F_\st(V_\st)-F_\theta(V_\theta)
 \leq R_\theta+1+\log\kappa_a.
\]
For each finite $n$, the stopped martingale in
\eqref{eq:score-Ito-power} is square-integrable; this is used only to justify
optional stopping, and its quadratic variation is not inserted into the
estimate.  Optional stopping gives
\[
 \E\left[\left.\int_\theta^\st\sum_iS_i^2\,\dd t
 \right|\F_\theta\right]
 \leq\frac{\kappa_a-1}{\log\kappa_a}
       (R_\theta+1+\log\kappa_a),
\]
which is the first assertion.  Since
$(\kappa_a-1)/\log\kappa_a=\kappa_a/\Lambda_a$ and
$\bar\delta=\alpha/(R_\theta+1)$, multiplication by $\bar\delta^2$ gives
\[
 \bar\delta^2\frac{\kappa_a-1}{\log\kappa_a}
 (R_\theta+1+\log\kappa_a)
 \lesssim
 \frac{\kappa_a}{\Lambda_a(R_\theta+1)}
 +\frac{\kappa_a-1}{(R_\theta+1)^2}.
\]
Averaging over $A$ proves \eqref{eq:power-energy-summary}.
\end{proof}

\section{Technical Results}\label{sec:technical}

\subsection{Joint filtration and reverse marginal}
\label{subsec:joint-filtration-technical}

Let $P^V_{r,t}$ denote the transition operator of the unperturbed reverse
process from time $r$ to time $t$.

\begin{lemma}[Reverse marginal in the joint filtration]
\label{lem:joint-filtration}
For every bounded terminal test $h:G\to\R$ and every
$\theta\leq t\leq\To$,
\begin{equation}\label{eq:joint-reverse-transition}
 \E[h(V_{\To})\mid\F_t]
 =P^V_{t,\To}h(V_t).
\end{equation}
For the reverse process already defined on $[0,T]$, the increments of the
driving Poisson random measures after $\To$ are independent of
$\F_{\To}$.  Consequently, for every bounded $\Phi:G\to\R$,
\begin{equation}\label{eq:joint-continued-transition}
 \E[\Phi(V_T)\mid\F_t]
 =P^V_{t,T}\Phi(V_t),
 \qquad \theta\leq t\leq\To.
\end{equation}
In particular, if
$H_t^\zeta(x)=P^V_{t,T}\1_{\{\zeta\}}(x)$, then
\begin{equation}\label{eq:joint-terminal-state}
 \P(V_T=\zeta\mid\F_t)=H_t^\zeta(V_t).
\end{equation}
\end{lemma}

\begin{proof}
The compensated-Poisson formula applied to the SDEs for $V$ and $W$ gives
\eqref{eq:power-joint-generator}--\eqref{eq:power-perturbation-generator}.
Indeed, if $Y_i<1$, the common rate is $Y_i/2$ and the $W$-only excess is
$(1-Y_i^{\bar\delta})/2=\delta_iS_i$.  If $Y_i\geq1$, the common rate is
$Y_i^{1-\bar\delta}/2$ and the $V$-only excess is
$(Y_i-Y_i^{1-\bar\delta})/2=-\delta_iS_i$.

Let $u(t,x)=P^V_{t,\To}h(x)$.  Then
$(\partial_t+\widetilde{\mathcal L}_t)u=0$ and
$u(\To,\cdot)=h$.  Regard $u$ as a function of $(x,y)$ that is
independent of $y$.  Every $y$-difference in
\eqref{eq:power-perturbation-generator} then vanishes, while
$\overline{\mathcal L}_t^0u=\widetilde{\mathcal L}_tu$.  Therefore
\[
 (\partial_t+\mathcal G_t)u(t,V_t,W_t)=0.
\]
The time-dependent Dynkin formula, valid for the bounded predictable
coefficients here, shows directly that $u(t,V_t)$ is an
$(\F_t)$-martingale.  Evaluating it at $\To$ proves
\eqref{eq:joint-reverse-transition}; no uniqueness theorem for a
martingale problem is needed.

For $v(t,x)=P^V_{t,T}\Phi(x)$, the same calculation on
$[\theta,\To]$ gives
\[
 v(t,V_t)=\E[v(\To,V_{\To})\mid\F_t].
\]
The increments of the original driving Poisson random measures on
$(\To,T]$ are independent of $\F_{\To}$, conditionally on
$V_{\To}$, and hence
$v(\To,V_{\To})=\E[\Phi(V_T)\mid\F_{\To}]$.  The tower property proves
\eqref{eq:joint-continued-transition}; taking
$\Phi=\1_{\{\zeta\}}$ yields \eqref{eq:joint-terminal-state}.
\end{proof}

This is the joint-filtration principle isolated in
\citet[Lemma~2.1]{XiangZhang2026}; here it follows directly from the common
Poisson construction.  Equation~\eqref{eq:joint-reverse-transition} is used
in the switched power-supermartingale and localized Duhamel arguments, while
\eqref{eq:joint-terminal-state} supplies the likelihood martingale for the
Boolean bridge and conditioned-generator calculation.

\subsection{Proof of Lemma~\ref{lem:bridge-identities}}\label{sec:bridge-identities}

Bayes' formula for the reversible heat kernel gives
\[
 H_t^\zeta(x)=\P(V_T=\zeta\mid V_t=x)
 =\frac{2^{-n}f(\zeta)\prod_j(1+\rho_tx_j\zeta_j)}{f_{T-t}(x)}.
\]
Taking the ratio at $x$ and $\sigma_i x$ proves
\eqref{eq:bridge-r-lambda}.  In the conditional law of
$V_{\To}$ given $(V_t,V_T)=(x,\zeta)$, the factor $f(\zeta)$ in the
joint endpoint density cancels.  The remaining two Boolean heat kernels
factor coordinatewise; hence the coordinates of $V_{\To}$ are
conditionally independent.  A one-coordinate Bayes calculation gives
\[
 \E[V_{\To}^{(i)}\mid V_t=x,V_T=\zeta]
 =\frac{\gamma_tx_i+a\zeta_i}{1+a\gamma_tx_i\zeta_i}.
\]
Under the synchronized bridge started from $(x,y)$, the product of the two
coordinates is preserved, so
$W_{\To}^{(i)}=x_iy_iV_{\To}^{(i)}$.  Its conditional mean is therefore
\[
 \E[W_{\To}^{(i)}\mid V_t=x,W_t=y,V_T=\zeta]
 =a_ty_i+b_tx_iy_i\zeta_i=m_t^{[i]}(x,y,\zeta).
\]
Conditional independence and multilinearity now give the explicit bridge
representation
\begin{equation}\label{eq:q-conditional-representation}
 q_t^\zeta(x,y)
 =\E[\phi(W_{\To})\mid V_t=x,W_t=y,V_T=\zeta]
\end{equation}
for the synchronized continuation.  Taking $y=x$ proves
\eqref{eq:bridge-starting-identity}.
These are the Boolean bridge formulas of
\citet[Lemma~9]{Chen2026}.  Multilinearity gives
\eqref{eq:bridge-diff-plus} and \eqref{eq:bridge-diff-minus}, as well as
$\Delta_i^{xy}q_t^\zeta=-2a_ty_i\partial_i\phi(m_t)$.

A direct substitution, with $\varepsilon=x_i\zeta_i$, gives
\[
 \frac{\lambda_{t,i}^\zeta b_t^2}{1-(m_t^{[i]})^2}
 =\frac{a^2(1-\gamma_t^2)}{(1-a^2)(1-a^2\gamma_t^2)}
 \leq\frac{a^2}{1-a^2},
\]
which proves \eqref{eq:bridge-b-control}.  This is the weighted bridge estimate
used in the proof of \citet[Lemma~11]{Chen2026}.

Finally, $q_t^\zeta$ is the conditional terminal expectation under the
synchronized Boolean bridge, so
$(\partial_t+\mathcal L_t^{0,\zeta})q_t^\zeta=0$.  The finite-state
carr\'e-du-champ identity and
$\Delta_i^{xy}q_t^\zeta=-2a_ty_i\partial_i\phi(m_t)$ prove
\eqref{eq:bridge-square-energy}.

\subsection{Proof of Lemma~\ref{lem:conditioned-power}}\label{sec:conditioned-power}

Conditioning by the space-time harmonic likelihood $H_t^\zeta(V_t)$ multiplies
every jump that flips the $V$-coordinate by $r_{t,i}^\zeta$ and leaves a
$W$-only jump unchanged.  Together with
$r_{t,i}^\zeta Y_i=\lambda_{t,i}^\zeta$, this gives
\eqref{eq:conditioned-power-generator}.  This is the predictable Doob-transform
calculation of \citet[Lemma~10]{Chen2026}.

It remains to prove \eqref{eq:conditioned-power-bound}.  
We first give the following bound on the size of the power perturbation.

\begin{lemma}[Size of the power perturbation]\label{lem:power-delta-bound}
      On $\mathcal E_\theta$, one has
      $
       0\leq\delta_i(t,x)\leq\Lambda_a\bar\delta,
       \text{ for all } \theta\leq t\leq\To, x\in G.
      $
      \end{lemma}
      
      \begin{proof}
      Fix $i,t,x$ and write $Y=Y_i(t,x)$.  If $0<Y<1$, then
      $1-Y^{\bar\delta}\leq-\bar\delta\log Y$, and hence
      \[
       \frac{1-Y^{\bar\delta}}{1-Y}\leq
       \bar\delta\frac{-\log Y}{1-Y}\leq \bar\delta\Lambda_a.
      \]
      The last inequality follows because $y\mapsto-\log y/(1-y)$ is decreasing on
      $(0,1)$ and $Y\geq\kappa_a^{-1}$.  If $Y>1$, then
      \[
       \frac{Y-Y^{1-\bar\delta}}{Y-1}
       \leq \bar\delta\frac{Y\log Y}{Y-1}\leq \bar\delta\Lambda_a,
      \]
      because $y\mapsto y\log y/(y-1)$ is increasing on $(1,\infty)$ and
      $Y\leq\kappa_a$.  The value at $Y=1$ follows by continuity.
      \end{proof}

On $\mathcal E_\theta^c$, one has $\bar\delta=0$ and
$\delta_i=0$, so \eqref{eq:conditioned-power-bound} is immediate.
Work on $\mathcal E_\theta$ and suppress the arguments $(t,x)$.  If
$S_i>0$, then $Y_i<1$ and Lemma~\ref{lem:power-delta-bound} gives
$
 \delta_i^2\leq\Lambda_a^2\bar\delta^2
 \leq\kappa_a\Lambda_a^2\bar\delta^2\lambda_{t,i}^\zeta,$
because $\lambda_{t,i}^\zeta\geq\kappa_a^{-1}$.  If $S_i=0$, then
$Y_i=1$, $\delta_i=\bar\delta$, and
$r_{t,i}^\zeta=\lambda_{t,i}^\zeta$ by
\eqref{eq:bridge-r-lambda}.  Hence
\[
 |r_{t,i}^\zeta\delta_i|^2
 =\bar\delta^2(\lambda_{t,i}^\zeta)^2
 \leq\kappa_a\bar\delta^2\lambda_{t,i}^\zeta
 \leq\kappa_a\Lambda_a^2\bar\delta^2\lambda_{t,i}^\zeta.
\]
Finally, if $S_i<0$, then $Y_i>1$ and
\[
 r_{t,i}^\zeta\delta_i
 =\lambda_{t,i}^\zeta
   \frac{1-Y_i^{-\bar\delta}}{Y_i-1}.
\]
Since $1-Y^{-\bar\delta}\leq\bar\delta\log Y
\leq\bar\delta(Y-1)$ for $Y>1$, one has
$r_{t,i}^\zeta\delta_i\leq \bar\delta\lambda_{t,i}^\zeta$, and hence
\[
 |r_{t,i}^\zeta\delta_i|^2
 \leq\kappa_a\bar\delta^2\lambda_{t,i}^\zeta
 \leq\kappa_a\Lambda_a^2\bar\delta^2\lambda_{t,i}^\zeta.
\]
This proves the stated uniform bound.

\bibliographystyle{plainnat}
\bibliography{references}

@misc{Chen2026,
  author        = {Chen, Yuansi},
  title         = {{Talagrand}'s convolution conjecture up to loglog via perturbed reverse heat},
  year          = {2025},
  howpublished  = {\emph{arXiv:2511.19374v2}},
  eprint        = {2511.19374},
  archiveprefix = {arXiv},
  primaryclass  = {math.PR}
}

@misc{Talagrand2016,
  author       = {Talagrand, Michel},
  title        = {Regularization from {$L^1$} by Convolution},
  year         = {2016},
  howpublished = {\url{https://michel.talagrand.net/prizes/convolution.pdf}}
}

@article{ball2013l1,
  title={{$L^1$}-smoothing for the Ornstein--Uhlenbeck semigroup},
  author={Ball, Keith and Barthe, Franck and Bednorz, Witold and Oleszkiewicz, Krzysztof and Wolff, Pawe{\l}},
  journal={Mathematika},
  volume={59},
  number={1},
  pages={160--168},
  year={2013},
  publisher={London Mathematical Society}
}

@article{EldanLee2018,
  author  = {Eldan, Ronen and Lee, James R.},
  title   = {Regularization under diffusion and anticoncentration of the information content},
  journal = {Duke Mathematical Journal},
  year    = {2018},
  volume  = {167},
  number  = {5},
  pages   = {969--993}
}

@article{Lehec2016,
  author  = {Lehec, Joseph},
  title   = {Regularization in {$L_1$} for the {Ornstein--Uhlenbeck} semigroup},
  journal = {Annales de la Facult{\'e} des sciences de Toulouse: Math{\'e}matiques},
  year    = {2016},
  volume  = {25},
  number  = {1},
  pages   = {191--204}
}

@book{ODonnell2014,
  author    = {O'Donnell, Ryan},
  title     = {Analysis of {Boolean} Functions},
  publisher = {Cambridge University Press},
  address   = {Cambridge},
  year      = {2014},
  isbn      = {978-1-107-03832-5}
}

@article{Talagrand1989,
  author  = {Talagrand, Michel},
  title   = {A conjecture on convolution operators, and a non-{Dunford--Pettis} operator on {$L^1$}},
  journal = {Israel Journal of Mathematics},
  year    = {1989},
  volume  = {68},
  number  = {1},
  pages   = {82--88}
}

@misc{XiangZhang2026,
  author        = {Xiang, Yanjin and Zhang, Zhihua},
  title         = {Layerwise terminal discrepancy in {Chen}'s reverse-heat coupling on the {Boolean} cube},
  year          = {2026},
  howpublished  = {\emph{arXiv:2606.04573v2}},
  eprint        = {2606.04573},
  archiveprefix = {arXiv},
  primaryclass  = {math.FA}
}

\end{document}